\documentclass[11pt]{amsart}

\usepackage{amssymb,amscd,amsmath,hyperref,color,enumerate}
\usepackage{ifthen}


\usepackage[margin=3cm]{geometry}

\newcommand{\N}{{\ensuremath{\mathbb{N}}}}
\newcommand{\Z}{{\ensuremath{\mathbb{Z}}}}
\newcommand{\Q}{{\ensuremath{\mathbb{Q}}}}
\newcommand{\C}{{\ensuremath{\mathbb{C}}}}

\newcommand{\K}{{\ensuremath{\mathbb{K}}}}

\newcommand{\F}{{\ensuremath{\mathbb{F}}}}

\usepackage[normalem]{ulem}

\newcommand{\stkout}[1]{\ifmmode\text{\sout{\ensuremath{#1}}}\else\sout{#1}\fi}

\DeclareMathOperator{\CIm}{Im}
\DeclareMathOperator{\supp}{supp}

\DeclareMathOperator{\diag}{diag}

\DeclareMathOperator{\idem}{Idem}
\DeclareMathOperator{\pot}{Pot}

\DeclareMathOperator{\chr}{char}

\DeclareFontFamily{U}{mathb}{\hyphenchar\font45}
\DeclareFontShape{U}{mathb}{m}{n}{
      <5> <6> <7> <8> <9> <10> gen * mathb
      <10.95> mathb10
      <12> <14.4> <17.28> <20.74> <24.88> mathb12
}{}
\DeclareSymbolFont{mathb}{U}{mathb}{m}{n}

\DeclareMathSymbol{\precneq}{3}{mathb}{"AC}

\newcommand{\Sph}{\ensuremath{\mathbb{S}}}

\newcommand{\ca}[1]{\ensuremath{\mathcal{#1}}}

\newtheorem{proposition}{Proposition}[section]
\newtheorem{lemma}[proposition]{Lemma}

\newtheorem{theorem}[proposition]{Theorem}

\newtheorem{corollary}[proposition]{Corollary}
\theoremstyle{definition}

\newtheorem{example}[proposition]{Example}

\newtheorem{remark}[proposition]{Remark}

\numberwithin{equation}{section}

\newlength{\leftstackrelawd}
\newlength{\leftstackrelbwd}
\def\leftstackrel#1#2{\settowidth{\leftstackrelawd}%
{${{}^{#1}}$}\settowidth{\leftstackrelbwd}{$#2$}%
\addtolength{\leftstackrelawd}{-\leftstackrelbwd}%
\leavevmode\ifthenelse{\lengthtest{\leftstackrelawd>0pt}}%
{\kern-.5\leftstackrelawd}{}\mathrel{\mathop{#2}\limits^{#1}}}

\begin{document}

\title[The structure of $k$-potents and mixed Jordan--power preservers]{The structure of $k$-potents and mixed Jordan--power preservers on matrix algebras}

\author{Ilja Gogi\'{c}, Mateo Toma\v{s}evi\'{c}}

\address{I.~Gogi\'c, University of Zagreb Faculty of Science, Department of Mathematics, Bijeni\v{c}ka 30, 10000 Zagreb, Croatia}
\email{ilja@math.hr, ilja.gogic@math.pmf.unizg.hr}
	
\address{M.~Toma\v{s}evi\'c, University of Zagreb Faculty of Science, Department of Mathematics, Bijeni\v{c}ka 30, 10000 Zagreb, Croatia}
\email{mateo.tomasevic@math.hr, mateo.tomasevic@math.pmf.unizg.hr}




\date{\today}

\begin{abstract}
Let $M_n(\mathbb{F})$ denote the algebra of $n \times n$ matrices over an algebraically closed field $\mathbb{F}$ of characteristic different from $2$. For $n \ge 2$, we classify all maps 
$\phi : M_n(\mathbb{F}) \to M_n(\mathbb{F})$ satisfying the mixed Jordan--power identity
$$
\phi(A^{k} \circ B) = \phi(A)^{k} \circ \phi(B), \quad \text{for all } A,B \in M_n(\mathbb{F}),
$$
where $\circ$ denotes the (normalized) Jordan product $A \circ B := \tfrac{1}{2}(AB + BA)$ and $k \in \mathbb{N}$. We show that every such map is either constant, taking a fixed $(k+1)$-potent value, or there exist an invertible matrix $T \in M_n(\mathbb{F})$, a ring monomorphism $\omega : \mathbb{F} \to \mathbb{F}$, and a $k$-th root of unity $\varepsilon \in \mathbb{F}$ such that $\phi$ takes one of the forms
$$
\phi(X) = \varepsilon\, T\, \omega(X)\, T^{-1} \quad \text{ or } \quad \phi(X) = \varepsilon\, T\, \omega(X)^{t}\, T^{-1},
$$
where $\omega(X)$ denotes the matrix obtained by applying $\omega$ entrywise to $X$, and $(\cdot)^{t}$ denotes matrix transposition. In particular, every nonconstant solution is necessarily additive. The classification relies fundamentally on the preservation of $(k+1)$-potents and their intrinsic structural properties.
\end{abstract}

\maketitle

\allowdisplaybreaks
\section{Introduction}

The study of multiplicative maps between rings and algebras has a long and substantial history, centered on the question of when multiplicativity alone enforces stronger algebraic behavior, most notably additivity. Foundational contributions were made by Rickart~\cite{Rickart} in 1948 and Johnson~\cite{Johnson} in 1958, who established positive results under suitable structural assumptions and highlighted the connection with the rigidity of the additive structure of rings. A major advance was achieved by Martindale~\cite{Martindale} in 1969, who proved that every bijective multiplicative map from a prime ring containing a nontrivial idempotent onto an arbitrary ring is automatically additive. In the same year, Jodeit and Lam~\cite{JodeitLam} classified all nondegenerate multiplicative self-maps of the matrix rings $M_n(\mathcal{R})$ over a principal ideal domain~$\mathcal{R}$, where nondegeneracy means that the map is not identically zero on the set of matrices with zero determinant. Specifically, for each such map $\phi: M_n(\ca{R}) \to M_n(\ca{R})$, either there exists a nonzero idempotent matrix $P \in M_n(\ca{R})$ such that $\phi - P$ is multiplicative and degenerate, or there exists an invertible matrix $T \in M_n(\ca{R})$ and a ring endomorphism $\omega$ of $\ca{R}$ such that $\phi$ takes one of the following two distinct forms:
$$
\phi(X) = T \, \omega(X) \, T^{-1} \quad \text{ or } \quad \phi(X) = T \, \omega(X)^* \, T^{-1},
$$
where $\omega(X)$ denotes the matrix obtained by applying $\omega$ entrywise, and $(\cdot)^*$ is the corresponding cofactor matrix. In particular, all bijective multiplicative self-maps on $M_n(\mathcal{R})$ are automatically additive and hence ring automorphisms. Taken together, these results expose a fundamental phenomenon: multiplicativity, when coupled with bijectivity, interacts strongly with the ambient ring structure, leaving little room for pathological nonadditive behavior. This rigidity principle later emerged as a central theme in preserver theory, especially in the context of matrix and operator algebras.

This viewpoint was later extended from ordinary ring homomorphisms to \emph{Jordan homomorphisms}, which originate as homomorphisms in the category of \emph{Jordan algebras}, introduced by Jordan, von Neumann, and Wigner \cite{Jordan} in the 1930s as algebraic models for quantum observables.
A \emph{Jordan algebra} is a nonassociative algebra $\ca{A}$ over a field equipped with a commutative multiplication $\diamond$ satisfying the \emph{Jordan identity}
$$
a^2 \diamond (a \diamond b) = a \diamond (a^2 \diamond b), \quad \text{ for all } a,b \in \ca{A}.
$$
Jordan rings are similarly defined. Every associative algebra (or ring) $\mathcal{A}$ naturally becomes a Jordan algebra (ring) when equipped with the \emph{Jordan product}
$$
a \diamond b := ab + ba.
$$
A Jordan algebra that is isomorphic to a Jordan subalgebra of an associative algebra is called a \emph{special Jordan algebra} (see e.g. \cite[Section~2.3]{HancheStormer}). Not all Jordan algebras are special; the Albert algebra~\cite{Albert} provides the classical example of an \emph{exceptional Jordan algebra}.

In an associative setting, a linear (or additive) map $\phi: \ca{A} \to \ca{B}$ between associative algebras (or rings) $\ca{A}$ and $\ca{B}$ is called a \emph{Jordan homomorphism} if it satisfies
$$
\phi(a \diamond b) = \phi(a) \diamond \phi(b), \quad \text{ for all } a, b \in \ca{A}.
$$
If the algebras (rings) are $2$-torsion-free, this is equivalent to requiring that $\phi$ preserves squares with respect to the underlying product of $\mathcal{A}$, that is,
$$
\phi(a^2) = \phi(a)^2, \quad \text{ for all } a \in \ca{A}.
$$
Moreover, if $\ca{A}$ is an algebra over a field $\F$ with $\chr(\F) \ne 2$, the Jordan structure is commonly (esp.\ in quantum mechanics contexts) induced by the \emph{normalized Jordan product}
\begin{equation}\label{eq:normJprodukt}
a\circ b :=\frac12(ab+ba),
\end{equation}
and Jordan homomorphisms are equivalently defined as $\circ$-multiplicative linear maps. The study of Jordan homomorphisms in associative settings is well established and constitutes a major research direction in abstract algebra, operator theory, and the theory of preservers.

The problem of automatic additivity for \emph{Jordan multiplicative maps} has been studied extensively since the early 2000s. This line of research focuses on maps that preserve the Jordan product, $\circ$ or $\diamond$, viewed as a natural analogue of the classical automatic additivity problem for multiplicative maps. A seminal contribution was made by Moln\'ar \cite{Molnar}, who characterized $\circ$-preserving bijections $\phi:\mathcal{A} \to \mathcal{B}$ between standard operator algebras $\mathcal{A}$ and $\mathcal{B}$ with $\dim(\mathcal{A})>1$, that is, subalgebras of bounded linear operators on a complex Banach space containing all finite-rank operators. In particular, all such maps are automatically additive. The finite-dimensional variant of Moln\'ar's result  asserts that any $\circ$-preserving bijection $\phi : M_n(\C)\to M_n(\C)$, $n \ge 2$, takes the form
\begin{equation}\label{eq:glavnieq}
\phi(X)=T\,\omega(X)\,T^{-1} \quad \text{ or } \quad \phi(X)=T\,\omega(X)^t\,T^{-1},
\end{equation}
where $T\in M_n(\C)$ is invertible, $\omega$ is a ring automorphism of $\C$ applied entrywise, and $(\cdot)^t$ denotes matrix transposition. Lu \cite{Lu} subsequently extended this result to additional classes of associative algebras. Subsequently, An and Hou \cite{AnHou} studied Jordan multiplicative bijections on the real Jordan algebra of all self-adjoint operators on a complex Hilbert space of dimension greater than $1$, proving that any $\diamond$- or $\circ$-preserving bijection is automatically additive and implemented via unitary or anti-unitary conjugation. The problem of automatic additivity has also been investigated in the context of abstract Jordan algebras. In particular, Ji \cite{Ji} showed that if $\mathcal{A}$ is a Jordan algebra containing an idempotent satisfying suitable Peirce-type structural conditions, then any bijection $\phi:\mathcal{A}\to \mathcal{B}$ to another Jordan algebra $\mathcal{B}$ that preserves the Jordan product is necessarily additive, thereby unifying a range of operator-algebraic and ring-theoretic results within a single framework.

In the setting of full matrix algebras, our recent work \cite{GogicTomasevic-AM} provides a complete classification of all  $\circ$-preserving maps $\phi : M_n(\mathbb{F}) \to M_n(\mathbb{F})$, where $\chr(\mathbb{F}) \neq 2$. We show that every such map is either constant (and hence equal to a fixed idempotent) or a Jordan ring monomorphism of the form~\eqref{eq:glavnieq}, where $T \in M_n(\mathbb{F})$ is invertible and $\omega : \mathbb{F} \to \mathbb{F}$ is a ring monomorphism \cite[Theorem~1.1]{GogicTomasevic-AM}.   This generalizes the finite-dimensional version of Molnár's theorem. Moreover, if $\phi : \mathcal{A} \to \mathcal{B}$ is any $\diamond$-preserving map between $\mathbb{F}$-algebras (with $\chr(\mathbb{F}) \neq 2$), then the simple normalization
$$
\psi(x) := 2\,\phi\left(\frac{x}{2}\right), \quad x \in \mathcal{A},
$$
transforms $\phi$ into a $\circ$-preserving map $\psi$.  This reduction allows us to obtain a corresponding classification of $\diamond$-preserving maps on $M_n(\mathbb{F})$, 
with the only difference that the constant map now takes the value $P/2$ for some idempotent $P \in M_n(\mathbb{F})$. In particular, this result shows that Jordan multiplicative maps on $M_n(\mathbb{F})$ are even more rigid than general multiplicative maps: every nonconstant map is necessarily additive. For the fields $\mathbb{R}$ and $\mathbb{C}$, this classification, together with the Jodeit--Lam theorem, was further extended in \cite{GogicTomasevic-LAMA} to injective multiplicative and Jordan multiplicative maps on \emph{structural matrix algebras}, i.e.\ subalgebras of $M_n(\mathbb{F})$ containing all diagonal matrices \cite{vanWyk} (see also \cite[Proposition~3.1]{GogicTomasevic-LAA}). Further details on structural matrix algebras, their automorphisms, and their Jordan monomorphisms can be found in \cite{Coelho, GogicTomasevic-LAA}.

A substantial body of literature also investigates multiplicative maps with respect to products other than the Jordan product, particularly various symmetrized or triple-product variants such as $(a,b) \mapsto aba$, $(a,b,c) \mapsto abc + cba$, and many others (see, e.g., \cite{AnHou, Kuzma, LesnjakSze, Lu2, Molnar2} and references therein). More recently, preservers of the product
\begin{equation}\label{eq:the product}
    a^{2} \circ b = \frac12(a^{2} b + b a^{2})
\end{equation}
were studied by Ghorbanipour and Hejazian~\cite{GhorbanipourHejazian}, who showed that any bijection $\phi : \mathcal{A} \to \mathcal{A}$ preserving \eqref{eq:the product}, where $\mathcal{A}$ is a real standard Jordan operator algebra acting on a Hilbert space of dimension at least $2$, is of the form $\phi = \pm \psi$, where $\psi$ is a Jordan automorphism of $\mathcal{A}$.

The product $(a,b)\mapsto a^2 \circ b$ plays an important role in the theory of JB-algebras. Recall that a \emph{JB-algebra} is a (typically real) Jordan algebra $(\mathcal{A},\circ)$ equipped with a complete submultiplicative norm $\|\cdot\|$  (that is, a Jordan Banach algebra) which additionally satisfies
$$
\|a\|^2 \le \|a^2 + b^2\|, \quad \text{ for all }  a,b \in \mathcal{A}.
$$
A central feature of JB-algebras is that many aspects of their structure, such as order, spectral theory, and the theory of ideals, are encoded in the associated quadratic operators (see, for example, \cite{Battaglia}). In particular, the fundamental quadratic mapping $U_a : \mathcal{A} \to \mathcal{A}$,
$$
U_{a}(b) := 2(a \circ b) \circ a - a^{2} \circ b,
$$
plays a key role in the analysis of orthogonality, annihilators, and quadratic ideals. Moreover, the product $(a,b)\mapsto a^2 \circ b$ was used in \cite{Battaglia} to define orthogonality in JB-algebras: elements $a,b \in \mathcal{A}$ are said to be orthogonal if $a^{2} \circ b = 0$. In the setting of special JB-algebras, such as the self-adjoint part of a $C^{*}$-algebra, the Jordan product is given by \eqref{eq:normJprodukt}. In this case, the expression $a^{2} \circ b$ agrees exactly with the symmetrized product appearing in \eqref{eq:the product}.

\smallskip

This motivates our study of \emph{mixed Jordan--power preservers}, namely, maps that preserve a more general Jordan-type product
$$
a^k \circ b= \frac{1}{2}(a^k b + b a^k),
$$
for arbitrary $k \in \N$. Our main result provides a complete classification of mixed Jordan--power preservers on full matrix algebras.
\begin{theorem}\label{thm:main}
Let $\F$ be an algebraically closed field of characteristic different from $2$, and let $n,k \in \N$ with $n \ge 2$. A map $\phi : M_n(\F) \to M_n(\F)$
satisfies the mixed Jordan--power identity
\begin{equation}\label{eq:a certain product}
\phi(A^k \circ B) = \phi(A)^k \circ \phi(B), \qquad \text{for all } A,B \in M_n(\F),
\end{equation}
if and only if one of the following holds:
\begin{itemize}
\item[(a)] $\phi$ is constant, taking a fixed $(k+1)$-potent value;
\item[(b)] there exist an invertible matrix $T \in M_n(\F)$, a ring monomorphism $\omega : \F \to \F$, and a $k$-th root of unity $\varepsilon \in \F$ such that
\begin{equation}\label{eq:form of phi}
\phi(X) = \varepsilon\, T\, \omega(X)\, T^{-1}
\quad \text{ or } \quad
\phi(X) = \varepsilon\, T\, \omega(X)^t\, T^{-1}.
\end{equation}
\end{itemize}
In particular, every nonconstant solution of \eqref{eq:a certain product} is additive.
\end{theorem}
This result extends all previously known Jordan-type additivity results on $M_n(\F)$, at least for algebraically closed fields with $\operatorname{char}(\F) \neq 2$. The proof of Theorem~\ref{thm:main}, presented in Section~\S\ref{sec:main}, follows the general strategy of \cite[Theorem~1.1]{GogicTomasevic-AM}. The classification relies on the fact that maps satisfying \eqref{eq:a certain product} preserve $(k+1)$-potents along with their ranks and natural partial order, and that they are orthoadditive on these elements. These notions will be introduced and discussed in the next section. A notable difference from the case $k=1$, which corresponds to idempotents, is that when $\operatorname{char}(\F)$ is positive and divides $k$, $(k+1)$-potents may fail to be diagonalizable (see Remark~\ref{rem:d-potents not diagonalizable}). This phenomenon requires more refined arguments in order to handle the mixed Jordan--power preservers effectively. We also note that linear and closely related variants of $k$-potent preservers on matrix and operator algebras have been studied in papers such as \cite{HouHou,SongCao,YouWang}.

\section{Notation and Preliminaries}\label{sec:prel}
We begin by introducing notation that will be used throughout the paper.  Let $\F$ be a fixed algebraically closed field of characteristic different from $2$.  We denote by $\mathbb{F}^\times$ the multiplicative group of nonzero elements of $\mathbb{F}$, and by $\mathbb{F}[x]$ the polynomial algebra in one variable over $\mathbb{F}$.

\smallskip

Let $n \in \N$ be a fixed positive integer.
\begin{itemize}
\item[--] By $[n]$ we denote the set $\{1,\ldots, n\}$ and by $\Delta_n$ the diagonal $\{(j,j) : j \in [n]\}$ of $[n]^2$.
\item[--] By $M_n=M_n(\mathbb{F})$ we denote the algebra of $n \times n$ matrices over $\F$ and by  $\ca{D}_n$ its subalgebra consisting of all diagonal matrices. 
\item[-- ]We denote by $I_n$ the identity matrix in $M_n$; when the dimension is clear from the context, we simply write $I$.
\item[--] The rank of a matrix $X \in M_n$ is denoted by $r(X)$.
\item[--] For matrices $X,Y \in M_n$, we write $X \propto Y$  to indicate that either  $X = Y = 0$, or they are both nonzero and collinear.
\item[--] As usual, for $i,j \in [n]$, by $E_{ij}\in M_n$  we denote the standard matrix unit with $1$ at the position $(i,j)$ and $0$ elsewhere. For a matrix $X = [x_{ij}]_{i,j=1}^n \in M_n$ we define its \emph{support} as 
$$
\supp X:=\{(i,j)\in [n]^2 : x_{ij}\ne 0\}.
$$ 
\item[--] Given a ring endomorphism $\omega$ of $\F$, we use the same symbol $\omega$ to denote the induced ring endomorphism of $M_n$, defined by applying $\omega$ to each entry of the corresponding matrix:  
$$\omega(X)=[\omega(x_{ij})]_{i,j=1}^n, \quad \text{ for all } X=[x_{ij}]_{i,j=1}^n\in M_n.$$
\item[--] For $k \in \N$, by $\Sph^1_k$ we denote the set of all $k$-th roots of unity in $\F$.
\end{itemize}

\medskip

We shall require the following elementary fact from \cite{GogicTomasevic-AM}:
\begin{lemma}[{\cite[Lemma~2.3]{GogicTomasevic-AM}}]\label{le:contains entire CxC}
    Let $n \ge 2$ and let $\ca{S} \subseteq [n]^2 \setminus \Delta_n$ be a nonempty subset. Suppose that for each $(i,j) \in \ca{S}$ we have:
    \begin{enumerate}[(a)]
        \item $(i,k) \in \ca{S}, \text{ for all } k \in [n]\setminus \{i\}$,
        \item $(\ell,j) \in \ca{S}, \text{ for all } \ell \in [n]\setminus \{j\}$,
        \item $(j,i) \in \ca{S}.$
    \end{enumerate}
    Then $\ca{S} = [n]^2 \setminus \Delta_n$.
\end{lemma}
Let $\mathcal{A}$ be an arbitrary associative $\mathbb{F}$-algebra. Denote by $\idem(\mathcal{A})$ the set of all idempotents of $\mathcal{A}$, equipped with the standard partial order $\le$ defined by
$$
p \le q \;\iff\; pq = qp = p, \qquad p, q \in \idem(\mathcal{A}).
$$
We shall make occasional use of the following elementary observation from~\cite{GogicTomasevic-LAMA}.
\begin{lemma}[{\cite[Lemma~2.1]{GogicTomasevic-LAMA}}]\label{le:Jordan product calculations}
Let $\ca{A}$ be an $\F$-algebra. For $p \in \idem(\ca{A})$ and an arbitrary $a \in \ca{A}$ we have:
    \begin{enumerate}[(a)]
    \item $p \circ a = 0$ if and only if $pa = ap = pap = 0$.
        \item $p \circ a = a$ if and only if $pa = ap = pap=a$.
    \end{enumerate}
\end{lemma}

\smallskip

As usual, an element $p\in\ca{A}$ satisfying $p^k=p$ for some $k\geq 2$ is
called a \emph{$k$-potent}; see \cite{HouHou,SongCao,YouWang}. Since $k$
already denotes the exponent in the mixed Jordan--power identity, we use $d$
for the potency parameter in this section and set $d=k+1$ in
Section~\ref{sec:main}.

Fix an integer $d\geq 2$. We denote the set of all $d$-potents in
$\ca{A}$ by $\pot_d(\ca{A})$. For each $p\in\pot_d(\ca{A})$, the element
$p^{d-1}$ is idempotent, and
$$
\idem(\ca{A})=\pot_2(\ca{A})\subseteq\pot_d(\ca{A}).
$$
For $p,q\in\pot_d(\ca{A})$, we write
$$
p\perp q
\quad\stackrel{\mathrm{def}}{\Longleftrightarrow}\quad
pq=qp=0,
$$
and say that $p$ and $q$ are \emph{orthogonal}. For every $1\le \ell\le d-1$,
\begin{equation}\label{eq: x^2 orthogonal}
 p\perp q \quad\Longleftrightarrow\quad p^\ell\perp q,
 \qquad p,q\in\pot_d(\ca{A}).
\end{equation}
We extend the standard partial order on idempotents to $\pot_d(\ca{A})$ by setting
$$
p\preceq q \quad\stackrel{\mathrm{def}}{\Longleftrightarrow}\quad pq=qp=p^2.
$$
Equivalently, $p^\ell q=qp^\ell=p^{\ell+1}$ for every $1\le \ell\le d-1$; in particular,
\begin{equation}\label{eq:preceq alternative definition}
 p\preceq q \quad\Longleftrightarrow\quad p^{d-1}q=qp^{d-1}=p.
\end{equation}
The order is directly related to orthogonality as follows:
\begin{equation}\label{eq:preceq orthogonal decomposition}
 p\preceq q
 \quad\Longleftrightarrow\quad
 q-p\in\pot_d(\ca{A})\ \text{ and }\ p\perp(q-p).
\end{equation}
Indeed, if $p\preceq q$ and $s:=q-p$, then $p\perp s$ and, since $q=p+s$, we have
$$
s^d=q^d-p^d=q-p=s.
$$
Conversely, if $q=p+s$ for some $s\in\pot_d(\ca{A})$ orthogonal to $p$, then $pq=qp=p^2$, so $p\preceq q$.

\begin{remark}
Note that $p\preceq q$ implies $p^{d-1}\le q^{d-1}$ as idempotents, but the converse fails in general. For example, let $p\in\ca{A}$ be a nonzero $d$-potent and choose $\omega\in\Sph_{d-1}^1\setminus\{1\}$, whenever such an element exists. Then $q:=\omega p$ satisfies $p^{d-1}=q^{d-1}$, whereas $p\not\preceq q$. If $\ca{A}$ is unital, then $p\preceq 1$ if and only if $p$ is an idempotent.
\end{remark}

\begin{example}
Let $P=\diag(p_1,\ldots,p_n)$ and $Q=\diag(q_1,\ldots,q_n)$ be diagonal $d$-potents in $M_n$. Then $p_j,q_j\in\Sph_{d-1}^1\cup\{0\}$ for all $j$, and
$$
P\preceq Q
\quad\Longleftrightarrow\quad
p_j\ne0\ \Longrightarrow\ q_j=p_j
\quad\text{for every }j\in[n].
$$
\end{example}

\begin{lemma}\label{le:preceq properties}
Let $\ca{A}$ be an $\F$-algebra.
\begin{enumerate}[(a)]
\item $\preceq$ is a partial order on $\pot_d(\ca{A})$.
\item If $p_1,\ldots,p_m\in\pot_d(\ca{A})$ are mutually orthogonal, then $p_1+\cdots+p_m\in\pot_d(\ca{A})$. If, in addition, $q\in\pot_d(\ca{A})$ satisfies $p_1,\ldots,p_m\preceq q$, then $p_1+\cdots+p_m\preceq q$.
\item For $p,q\in\pot_d(\ca{A})$,
$$
p\preceq q \quad\Longleftrightarrow\quad p^{d-1}\circ q=p.
$$
\end{enumerate}
\end{lemma}
\begin{proof}
\begin{enumerate}[(a)]
\item Reflexivity is immediate. For antisymmetry, suppose that $p\preceq q$ and $q\preceq p$. By \eqref{eq:preceq orthogonal decomposition}, write $q=p+s$ with $s\in\pot_d(\ca{A})$ and $p\perp s$. Since the two order relations give $p^2=pq=qp=q^2$, we obtain
$$
p^2=q^2=(p+s)^2=p^2+s^2,
$$
so $s^2=0$ and hence $s=s^d=0$. Thus $p=q$.

For transitivity, suppose that $p\preceq q\preceq t$. Write $q=p+s$ and $t=q+u$ as in \eqref{eq:preceq orthogonal decomposition}. Since $q\perp u$ and $p=p^{d-1}q=qp^{d-1}$, we have $p\perp u$; consequently $s=q-p$ is also orthogonal to $u$. Therefore $s+u$ is a $d$-potent orthogonal to $p$, and
$$
t=p+(s+u).
$$
Another application of \eqref{eq:preceq orthogonal decomposition} yields $p\preceq t$.

\item Orthogonality eliminates all mixed terms, so
$$
(p_1+\cdots+p_m)^d=p_1^d+\cdots+p_m^d=p_1+\cdots+p_m.
$$
If also $p_j\preceq q$ for every $j$, then, with $p:=p_1+\cdots+p_m$,
$$
pq=qp=p_1^2+\cdots+p_m^2=p^2,
$$
and hence $p\preceq q$.

\item If $p\preceq q$, then \eqref{eq:preceq alternative definition} gives
$$
p^{d-1}\circ q=\frac12(p^{d-1}q+qp^{d-1})=p.
$$
Conversely, suppose that $p^{d-1}\circ q=p$. Multiplication by $2p^{d-1}$ from the left and from the right gives
$$
p^{d-1}q+p^{d-1}qp^{d-1}=2p,
\qquad
p^{d-1}qp^{d-1}+qp^{d-1}=2p.
$$
Thus $p^{d-1}q=qp^{d-1}$, and consequently both are equal to $p$. The conclusion follows from \eqref{eq:preceq alternative definition}.
\end{enumerate}
\end{proof}

\begin{remark}\label{re:maximal element}
Let $\ca{A}$ be a unital $\F$-algebra. Note that a $d$-potent $p$ is a maximal element of the poset $(\pot_d(\ca{A}),\preceq)$ if and only if it is invertible in $\ca{A}$. Indeed, if $p$ is invertible and $p\preceq q$, then $pq=p^2$ gives $q=p$. Conversely, if $p$ is maximal, then $1-p^{d-1}$ is an idempotent orthogonal to $p$, so
$$
q:=p+(1-p^{d-1})\in\pot_d(\ca{A})
\quad\text{and}\quad p\preceq q.
$$
Maximality yields $p=q$, hence $p^{d-1}=1$ and $p$ is invertible.
\end{remark}

Given an $\mathbb{F}$-algebra $\mathcal{A}$ and an additive (commutative) semigroup $\mathcal{S}$, we say that a map $\psi : \pot_d(\mathcal{A}) \to \mathcal{S}$ is \emph{orthoadditive} if 
$$
p \perp q \, \implies \, \psi(p+q) = \psi(p) + \psi(q), \quad \text{ for all } p,q \in \pot_d(\mathcal{A}).
$$
A simple inductive argument combined with Lemma \ref{le:preceq properties} (b), shows that for all $m \in \N$  and $p_1,\ldots,p_m \in \pot_d(\mathcal{A})$ we have
$$
p_1,\ldots,p_m \text{ pairwise orthogonal} \, \implies \, \psi(p_1+\cdots+p_m) = \psi(p_1) + \cdots + \psi(p_m).
$$

\begin{lemma}\label{le:d-potent rank}
\phantom{x}
\begin{enumerate}[(a)]
\item If $P,Q\in\pot_d(M_n)$ and $P\preceq Q$, then $r(P)\le r(Q)$, with equality if and only if $P=Q$.
\item The matrix rank is orthoadditive on $\pot_d(M_n)$.
\item For every $P\in\pot_d(M_n)$ and $j\in\N$, one has $r(P^j)=r(P)$.
\end{enumerate}
\end{lemma}
\begin{proof}
We first prove (b). If $P\perp Q$, then $\CIm(P)\subseteq\ker(Q)$ and $\CIm(Q)\subseteq\ker(P)$, so $\CIm(P)\cap\CIm(Q)=\{0\}$. Moreover,
$$
\CIm(P+Q)=\CIm(P)\dotplus\CIm(Q).
$$
Indeed, the inclusion from left to right is immediate. Conversely, if $x\in\CIm(P)$ and $y\in\CIm(Q)$, then $P^{d-1}x=x$ and $Q^{d-1}y=y$, whence
$$
(P+Q)^{d-1}(x+y)=x+y.
$$
Thus $r(P+Q)=r(P)+r(Q)$.

For (a), if $P\preceq Q$, then \eqref{eq:preceq orthogonal decomposition} gives $Q=P+R$ for some $R\in\pot_d(M_n)$ with $P\perp R$. By (b),
$$
r(Q)=r(P)+r(R)\ge r(P),
$$
and equality holds exactly when $R=0$, equivalently when $P=Q$.

Finally, for (c), the identity $P^{j+d-1}=P^j$ for every $j\ge1$ shows that it suffices to consider $1\le j<d$. Since $P^d=P$,
$$
\CIm(P)\supseteq\CIm(P^2)\supseteq\cdots\supseteq\CIm(P^j)
\supseteq\cdots\supseteq\CIm(P^d)=\CIm(P),
$$
so equality holds throughout.
\end{proof}

\begin{remark}\label{rem:d-potents not diagonalizable}
Suppose that $n \ge 2$. All matrices in $\pot_d(M_n)$ are diagonalizable if and only if $\chr(\F)$ does not divide $d-1$. Indeed, suppose that $\chr(\F)$ does not divide $d-1$, and let $P \in \pot_d(M_n)$ be arbitrary. We have
        $$P^{d} = P \, \implies \, P(P^{d-1}-I) = 0,$$
        which implies that the minimal polynomial $m_P \in \F[x]$ of $P$ necessarily divides 
        $$q(x):=x(x^{d-1}-1).$$ 
        Since $\F$ is algebraically closed, to show that $P$ is diagonalizable in $M_n$, it suffices to show that $q$ (and consequently $m_P$) is square-free, meaning that it decomposes into a product of distinct linear factors. We have
        $q'(x) = dx^{d-1}-1$. Assuming that $x_0 \in \F$ satisfies $q(x_0) = q'(x_0) = 0$ would lead to $d-1=0$, which is a contradiction.

        Conversely, suppose that $\chr(\F) \mid (d-1)$. For $n=2$, the matrix 
        $$
        A := \begin{bmatrix}
            1  & 1 \\ 0 & 1
        \end{bmatrix} \in M_2
        $$ 
        satisfies $$
        A^{d-1} = \begin{bmatrix}
            1  & d-1 \\ 0 & 1
        \end{bmatrix} = I \implies  A^d = A
        $$
        so $A$ is a $d$-potent in $M_2$. On the other hand, the minimal polynomial of $A$ is $m_A(x) = (x-1)^2 \in \F[x]$, which is not square-free in $\F[x]$. Therefore, $A$ is not diagonalizable in $M_2$. When $n > 2$, a similar conclusion follows by considering the matrix $\diag(A,I_{n-2}) \in M_n$.

        Additionally, the matrix $A$ demonstrates other interesting properties of the order $\preceq$. It is not difficult to show directly that any matrix $B \in M_2$ satisfies $AB=BA=B^2$ if and only if $B \in \{0,A\}$. In particular, $A$ is a minimal element of the poset $(\pot_d(M_2)\setminus\{0\}, \preceq)$, even though its rank is $2$. It is also a consequence that $A$ \emph{cannot} be written in the form $A = P+Q$ for some nonzero mutually orthogonal $d$-potents $P,Q \in \pot_d(M_2)$. On the other hand, by Remark \ref{re:maximal element}, $A$ is also a maximal element of the poset $(\pot_d(M_2), \preceq)$.
\end{remark}

We have the following generalization of \cite[Proposition~2.2]{GogicTomasevic-LAMA}, but with a more streamlined proof.
\begin{proposition}\label{p:M_n is Jordans simple}
Let $k,n \in \N$. Suppose that $\ca{J} \subseteq M_n$ is a nonempty subset which satisfies the following property:
\begin{equation}\label{eq:nonlinear Jordan ideal}
 (\text{for all } A,B \in M_n) \, ((A \in \ca{J} \text{ or } B \in \ca{J}) \, \implies \, (A^k \circ B \in \ca{J})).
\end{equation}
Then either $\ca{J} = \{0\}$ or $\ca{J} = M_n$.
\end{proposition}

\begin{proof}
Since $\F$ is algebraically closed, $\ca{J}$ is closed under scalar multiplication. Indeed, if $X\in\ca{J}$ and $\lambda\in\F$, choose $\mu\in\F$ with $\mu^k=\lambda$; then
$$
\lambda X=(\mu I_n)^k\circ X\in\ca{J}.
$$
For $1\le p\le n$, set
$$
\ca{J}_p:=\left\{X\in M_p:\begin{bmatrix}X&0\\0&0\end{bmatrix}\in\ca{J}\right\}.
$$
Under the natural upper-left-corner embedding of $M_p$ into $M_n$, property \eqref{eq:nonlinear Jordan ideal} passes immediately from $\ca{J}$ to $\ca{J}_p$.

Assume that $\ca{J}\ne\{0\}$. We first show that $\ca{J}$ contains every diagonal matrix unit. Fix a nonzero matrix $X=[x_{ij}]\in\ca{J}$. If $x_{ij}\ne0$ for some distinct $i,j$, then
\begin{equation}\label{eq:XijEij + XjiEji}
4\bigl((X\circ E_{ii}^k)\circ E_{jj}^k\bigr)
=x_{ij}E_{ij}+x_{ji}E_{ji}\in\ca{J}.
\end{equation}
The matrix
$$
H_{ij}:=2\left(E_{ii}+\frac1{x_{ij}}E_{ji}-E_{jj}\right)
$$
is diagonalizable, with nonzero eigenvalues $2$ and $-2$. Hence it has a $k$-th root in $M_n$. A direct calculation gives
$$
(x_{ij}E_{ij}+x_{ji}E_{ji})\circ H_{ij}=E_{ii}+E_{jj}\in\ca{J},
$$
and therefore
$$
E_{ii}=E_{ii}^k\circ(E_{ii}+E_{jj})\in\ca{J}.
$$
If $X$ is diagonal, choose $i \in [n]$ with $x_{ii}\ne0$. Then
$$
x_{ii}E_{ii}=E_{ii}^k\circ X\in\ca{J},
$$
so again $E_{ii}\in\ca{J}$ by scalar closure.

Now suppose that $E_{ii}\in\ca{J}$ and let $j\ne i$. Then
$$
E_{ji}=E_{ii}^k\circ(2E_{ji})\in\ca{J}.
$$
Applying the preceding off-diagonal construction to $X=E_{ji}$ yields $E_{ii}+E_{jj}\in\ca{J}$, and hence
$$
E_{jj}=E_{jj}^k\circ(E_{ii}+E_{jj})\in\ca{J}.
$$
Thus $E_{11},\ldots,E_{nn}\in\ca{J}$. In particular, $\ca{J}_p\ne\{0\}$ for every $p \in [n]$.

We complete the proof by induction on $n$. For $n=1$, scalar closure and $\ca{J}\ne\{0\}$ give $\ca{J}=M_1$. Let $n\ge2$ and assume the assertion is known in all smaller dimensions. Then $\ca{J}_p=M_p$ for every $p<n$. It remains to prove that $I_n\in\ca{J}$, because this implies
$$
B=I_n^k\circ B\in\ca{J},\qquad B\in M_n.
$$

If $n=2m$, then $I_m\in\ca{J}_m$, so
$$
N:=
\begin{bmatrix}0&I_m\\0&0\end{bmatrix}=\begin{bmatrix}I_m&0\\0&0\end{bmatrix}^{\!k}
\circ
\begin{bmatrix}0&2I_m\\0&0\end{bmatrix}
\in\ca{J}.
$$
Set
$$
H:=\begin{bmatrix}2I_m&0\\2I_m&-2I_m\end{bmatrix}.
$$
The matrix $H$ is diagonalizable and therefore has a $k$-th root. Since $H \circ N = I_n$, we obtain $I_n\in\ca{J}$.

If $n=2m+1$, then $I_{m+1}\in\ca{J}_{m+1}$, and hence
$$
N:= \begin{bmatrix}0&0&I_m\\0&1&0\\0&0&0\end{bmatrix}
=
\begin{bmatrix}I_m&0&0\\0&1&0\\0&0&0\end{bmatrix}^{k}
\circ
\begin{bmatrix}0&0&2I_m\\0&1&0\\0&0&0\end{bmatrix} \in\ca{J}.
$$
Set
$$
H:=\begin{bmatrix}2I_m&0&0\\0&1&0\\2I_m&0&-2I_m\end{bmatrix}.
$$
Its minimal polynomial is $(x-1)(x-2)(x+2)$ if $\chr(\F)\ne3$, and $(x-1)(x+1)$ if $\chr(\F)=3$. Thus $H$ is diagonalizable and has a $k$-th root. Since $H\circ N=I_n$, we again obtain $I_n\in\ca{J}$. This completes the induction.

\end{proof}

Just before the final step in the proof of Theorem~\ref{thm:main}, we show that any nonconstant map 
$\phi : M_n \to M_n$ satisfying \eqref{eq:a certain product} must take the form \eqref{eq:form of phi} on all matrices $X \in M_n$ that are $k$-th powers.  
The proof of the theorem is then completed with the following auxiliary result.
\begin{lemma}\label{le:iterated R} Let $k,n \in \N$. Denote
     $$\ca{R}_0^n := \{A^k : A \in M_n\}$$
    and for $\ell \ge 1$ denote
    $$\ca{R}_{\ell}^n := \{A^k \circ B : A,B \in \ca{R}_{\ell-1}^n\}.$$
    Then $$\ca{R}^n := \bigcup_{\ell \ge 0} \ca{R}_{\ell}^n = M_n.$$
\end{lemma}
\begin{proof}

We first record three closure properties.
\begin{enumerate}[(1)]
\item The sequence $(\ca{R}_\ell^n)_{\ell\ge0}$ is increasing. Indeed, $I_n\in\ca{R}_\ell^n$ for every $\ell$, and hence
$$
\ca{R}_\ell^n=\{I_n^k\circ B:B\in\ca{R}_\ell^n\}\subseteq\ca{R}_{\ell+1}^n.
$$
Consequently, $A,B\in\ca{R}^n$ implies $A^k\circ B\in\ca{R}^n$.

\item Direct sums preserve the construction. Since powers and Jordan products are computed blockwise, an induction on $\ell$ gives
$$
\diag(\ca{R}_\ell^p,\ca{R}_\ell^q)\subseteq\ca{R}_\ell^{p+q}, \qquad \ell\ge0.
$$
Therefore $\diag(\ca{R}^p,\ca{R}^q)\subseteq\ca{R}^{p+q}$.

\item Similarities preserve the construction. For every invertible $S\in M_n$,
$$
(SAS^{-1})^k=SA^kS^{-1},
\qquad
(SAS^{-1})\circ(SBS^{-1})=S(A\circ B)S^{-1}.
$$
Thus induction on $\ell$ yields
$$
S\ca{R}_\ell^nS^{-1}=\ca{R}_\ell^n,
\qquad \ell \ge 0,
$$
and hence $S\ca{R}^nS^{-1}=\ca{R}^n$.
\end{enumerate}

    In view of (3), as the underlying field $\F$ is algebraically closed, to prove the lemma it suffices to establish that $\ca{R}^n$ contains all Jordan matrices. In addition, by invoking (2), since every Jordan matrix is block-diagonal having Jordan blocks $$J_r(\lambda) := \begin{bmatrix} \lambda & 1 & \cdots & 0 & 0 \\
    0 & \lambda & \cdots & 0 & 0\\
    \vdots & \vdots & \ddots & \vdots & \vdots \\
    0 & 0 & \cdots & \lambda & 1 \\
    0 & 0 & \cdots & 0 & \lambda\end{bmatrix} \in M_r, \quad \lambda \in \F$$
    on the diagonal, it suffices to show that $J_r(\lambda) \in \ca{R}^r$ for all $\lambda \in \F, r \in \N$. 

    First we focus on the case $\lambda = 0$. By induction on $r$ we prove that $J_r(0) \in \ca{R}^r$. For $r = 1$, we have $J_1(0) = 0 \in \ca{R}_0^1$. Let $r \ge 2$ and assume that $J_{j}(0) \in \ca{R}^{j}$ for all $j \in [r-1]$. Define $A_r, B_r \in M_r$ as
    \begin{align*}
        A_r &:= 
\begin{bmatrix}
 2 & 0 & -1 & 0 & 0 & \cdots & 0 & 0 \\
 0 & 0 & 1 & 0 & 0 & \cdots & 0 & 0\\
 0 & 0 & 0 & 1 & 0 & \cdots & 0 & 0\\
 0 & 0 & 0 & 0 & 1 & \cdots & 0 & 0\\
     0 & 0 & 0 & 0 & 0 & \cdots & 0 & 0\\
 \vdots & \vdots & \vdots & \vdots & \vdots & \ddots & \vdots & \vdots\\ 
 0 & 0 & 0 & 0 & 0 & \cdots & 0 & 1 \\
 0 & 0 & 0 & 0 & 0 & \cdots & 0 & 0 \\
\end{bmatrix} = 2E_{11}-E_{13}+\sum_{2 \le j \le r-1} E_{j(j+1)},\\
B_r &:= \begin{bmatrix}
 0 & 1 & 0 & 0 & 0 & \cdots & 0 & 0 \\
 0 & 1 & 0 & 0 & 0 & \cdots & 0 & 0\\
 0 & 0 & 1 & 0 & 0 & \cdots & 0 & 0\\
  0 & 0 & 0 & 1 & 0 & \cdots & 0 & 0\\
 \vdots & \vdots & \vdots & \vdots & \vdots & \ddots & \vdots & \vdots\\ 
 0 & 0 & 0 & 0 & 0 & \cdots & 1 & 0 \\
 0 & 0 & 0 & 0 & 0 & \cdots & 0 & 1 \\
\end{bmatrix} = I_r-E_{11}+E_{12}.
    \end{align*}
    (For small $r$, the displayed matrices are truncated to their upper-left $r\times r$ blocks, and terms involving an index greater than $r$ are omitted.)
    
    One easily checks that the Jordan form of $A_r$ is  $$\diag(J_1(2), J_{r-1}(0)) \in M_r$$
    which is in $\ca{R}^r$ by the induction hypothesis. Therefore, $A_r \in \ca{R}^r$ as well. On the other hand, we have $B_r \in \idem(M_r) \subseteq \ca{R}_0^r$.
    Finally, it is straightforward to compute that
    $$J_r(0) = A_r \circ B_r = A_r \circ B_r^k \in \ca{R}^r,$$
    which completes the inductive step.

   Now we focus on the case $\lambda=1$. The case $r=1$ is immediate. For $r\ge2$, we use the same inductive approach as above and consider the matrices $C_r,D_r\in M_r$ given below, with terms involving an index greater than $r$ omitted:
    \begin{align*}
        C_r &:= 
\begin{bmatrix}
 0 & 2 & 0 & 0 & 0 & \cdots & 0 & 0 \\
 0 & 0 & 2 & 0 & 0 & \cdots & 0 & 0\\
 0 & 0 & 1 & 1 & 0 & \cdots & 0 & 0\\
 0 & 0 & 0 & 1 & 1 & \cdots & 0 & 0\\
 0 & 0 & 0 & 0 & 1 & \cdots & 0 & 0\\
 \vdots & \vdots & \vdots & \vdots & \vdots & \ddots & \vdots & \vdots\\ 
 0 & 0 & 0 & 0 & 0 & \cdots & 1 & 1 \\
 0 & 0 & 0 & 0 & 0 & \cdots & 0 & 1 \\
\end{bmatrix} = J_r(1)+ E_{12}+E_{23} -E_{11}-E_{22},\\
D_r &:= \begin{bmatrix}
 1 & 0 & 0 & 0 & 0 & \cdots & 0 & 0 \\
 1 & 0 & 0 & 0 & 0 & \cdots & 0 & 0\\
 0 & 0 & 1 & 0 & 0 & \cdots & 0 & 0\\
  0 & 0 & 0 & 1 & 0 & \cdots & 0 & 0\\
 \vdots & \vdots & \vdots & \vdots & \vdots & \ddots & \vdots & \vdots\\ 
 0 & 0 & 0 & 0 & 0 & \cdots & 1 & 0 \\
 0 & 0 & 0 & 0 & 0 & \cdots & 0 & 1 \\
\end{bmatrix} = I_r-E_{22}+E_{21}.
    \end{align*}
    
A direct computation shows that $C_r$ is similar to $\diag(J_2(0),J_{r-2}(1))$, where the second block is omitted when $r=2$. By the already established case $\lambda=0$, by the induction hypothesis when $r>2$, and by the direct-sum property (2), this block-diagonal matrix belongs to $\ca{R}^r$. Therefore, $C_r\in\ca{R}^r$. On the other hand, we have $D_r \in \idem(M_r) \subseteq \ca{R}_0^r$.
    Finally, it is straightforward to compute that
    $$J_r(1) = C_r \circ D_r = C_r \circ D_r^k \in \ca{R}^r,$$
    which completes the inductive step.

    For general $\lambda \in \F^\times$, let $\mu\in\F$ be a $k$-th root of $\lambda$. Then $\mu I_r \in \ca{R}_0^r$, and we have
    $$\ca{R}^r \ni (\mu I_r)^k \circ J_r(1) = \lambda J_r(1).$$
    Since $\lambda J_r(1)$ is similar to $J_r(\lambda)$, by (3) it follows $J_r(\lambda) \in \ca{R}^r$ as well.
\end{proof}

\section{Proof of Theorem~\ref{thm:main}}\label{sec:main}
Fix $k,m,n\in\mathbb{N}$ throughout this section. We apply the results of Section~\ref{sec:prel} with $d=k+1$. Our proof of Theorem~$\ref{thm:main}$ follows the strategy of \cite[Section~\S3]{GogicTomasevic-AM}, with the necessary adjustments. While some arguments could be abbreviated, we include the details for completeness. We begin by establishing several auxiliary lemmas, the first of which is an immediate consequence of Proposition~$\ref{p:M_n is Jordans simple}$.

\begin{lemma}\label{le:zero map}
Let $\mathcal{A}$ be an arbitrary $\mathbb{F}$-algebra, and let $\phi: M_n \to \mathcal{A}$ be a map satisfying \eqref{eq:a certain product}. If there exists a nonzero matrix $X \in M_n$ such that $\phi(X) = 0$, then $\phi$ is identically zero.
\end{lemma}
\begin{proof}
Let 
$$
\mathcal{J} := \phi^{-1}(0) = \{X \in M_n : \phi(X) = 0\}.
$$ 
It is clear that $\mathcal{J}$ satisfies the hypotheses of Proposition~\ref{p:M_n is Jordans simple}. Indeed, suppose that $A \in \ca{J}$ and $B \in M_n$. Then
$$
\phi(A^k \circ B) = {\underbrace{\phi(A)}_{=0}}^k \circ \phi(B) = 0, \qquad \phi(A \circ B^k) = {\underbrace{\phi(A)}_{=0}} \circ  \phi(B)^k = 0
$$
so $A^k \circ B, A\circ B^k \in \ca{J}$. Moreover, as $0 \neq X \in \mathcal{J}$ we conclude that $\mathcal{J} = M_n$, which implies that $\phi$ is identically zero.  
\end{proof}
The following lemma is a $(k+1)$-potent variant of \cite[Lemma~3.2]{GogicTomasevic-AM}.

\begin{lemma}\label{le:basic properties II}
Let $\phi: M_n \to M_m$ be an arbitrary map satisfying \eqref{eq:a certain product}. Then the following statements hold:
\begin{enumerate}
    \item[(a)] $\phi$ preserves $(k+1)$-th powers, i.e. 
    $$\phi(X^{k+1}) = \phi(X)^{k+1}, \quad \text{ for all }X \in M_n.$$
    In particular, $\phi$  preserves $(k+1)$-potents, that is $\phi(\pot_{k+1}(M_n))\subseteq \pot_{k+1}(M_m)$.
    \item[(b)] For $P,Q \in \pot_{k+1}(M_n)$, we have 
    $$P \preceq Q \, \implies \, \phi(P) \preceq \phi(Q).$$
\end{enumerate}
Suppose now that $\phi(0) = 0$ but $\phi \ne 0$. Then:
\begin{enumerate}
    \item[(c)] For $P,Q \in \pot_{k+1}(M_n)$ we have $P \perp Q \implies \phi(P) \perp \phi(Q)$.
    \item[(d)] For each nonzero $P \in \pot_{k+1}(M_n)$ we have $r(\phi(P)) \ge r(P)$. Consequently, $m \ge n$.
\end{enumerate}
Further, if $m=n$, then:
\begin{enumerate}
    \item[(e)] For each $P \in \pot_{k+1}(M_n)$ we have $r(\phi(P)) = r(P)$.
    \item[(f)] The restriction $\phi|_{\pot_{k+1}(M_n)}: \pot_{k+1}(M_n)\to \pot_{k+1}(M_n)$ is orthoadditive.
    \item[(g)] Suppose that $P_1,\ldots,P_r \in \pot_{k+1}(M_n)$ are mutually orthogonal and let $\lambda_1,\ldots,\lambda_r \in \F$. Then
    $$\phi\left(\sum_{j=1}^r \lambda_j P_j\right) = \sum_{j=1}^r \phi(\lambda_j P_j).$$
\end{enumerate}
\end{lemma}
\begin{proof}
    \begin{enumerate}[(a)]
        \item For $X \in M_n$ we have
        $$\phi(X^{k+1}) = \phi(X^k \circ X) = \phi(X)^k \circ \phi(X) = \phi(X)^{k+1}.$$
        \item If $P \preceq Q$, then by Lemma \ref{le:preceq properties} (c) we have $$\phi(P) = \phi(P^k \circ Q) = \phi(P)^k \circ \phi(Q).$$
By the same lemma, the latter is equivalent to  $\phi(P) \preceq \phi(Q)$.

\item We have
    $$\phi(P)^k \circ \phi(Q) = \phi(P^k \circ Q) = \phi(0) = 0.$$
   As $\phi(P)^k \in \idem(M_m)$, by Lemma \ref{le:Jordan product calculations} (a) and \eqref{eq: x^2 orthogonal}, we have $\phi(P) \perp \phi(Q)$.
\item Let $P \in \pot_{k+1}(M_n)$ be an arbitrary $(k+1)$-potent of rank $r \ge 1$. If $\chr(\F)$ does not divide $k$, then $P$ is diagonalizable (by Remark \ref{rem:d-potents not diagonalizable}) and we can proceed exactly as in the proof of \cite[Lemma~3.2~(d)]{GogicTomasevic-AM}. We here provide an argument which is independent of $\chr(\F)$. Since $P^k$ is an idempotent in $M_n$, it is therefore diagonalizable. Hence, there exist mutually orthogonal rank-one idempotents $Q_1, \ldots, Q_r \in \idem(M_n)$ such that 
$$P^k=Q_1+\cdots + Q_r.$$  
Since $\phi\neq 0$, Lemma~\ref{le:zero map} guarantees that $\phi$ does not annihilate any nonzero matrix. In particular, by (a), each $\phi(Q_{j})$ is a nonzero $(k+1)$-potent. It follows
$$
Q_{1},\ldots,Q_{r} \le P^{k}
\stackrel{(b)}{\implies}
\underbrace{\phi(Q_{1}),\ldots,\phi(Q_{r})}_{\text{mutually orthogonal by (c)}} \preceq \, \, \phi(P^{k}).
$$
Lemma~\ref{le:preceq properties} (b) then yields
$$
\phi(Q_{1})+\cdots+\phi(Q_{r}) \preceq \phi(P^{k}).
$$
By Lemma \ref{le:d-potent rank} we conclude $r(\phi(P^k)) \ge r$.
Notice now that
$$
\phi(P)^{k} \circ \phi(P^{k}) = \phi(P^{k} \circ P^{k})  = \phi(P^{2k}) = \phi(P^{k}).
$$
Since $\phi(P)^{k}$ is an idempotent, by Lemma~\ref{le:Jordan product calculations} (b) and Lemma \ref{le:d-potent rank} (c) we conclude
$$
\phi(P^{k})\,\phi(P)^{k} = \phi(P^{k}) \, \implies \, r \le r(\phi(P^{k})) \le r(\phi(P)^{k}) = r(\phi(P)).
$$
In particular, by considering $\phi(I_n)$, it follows $m \ge n$.
\item Let $P\in \pot_{k+1}(M_n)$ be an arbitrary $(k+1)$-potent. Then $P^{k} \in \idem(M_n)$ so $I-P^k \in \idem(M_n)$ as well, which is also orthogonal to $P$. By (c) we have $\phi(P) \perp \phi(I-P^k)$ and, since $r(P) = r(P^k)$ (by Lemma~\ref{le:d-potent rank} (c)), we have
$$n = r(P) + r(I-P^k) \stackrel{(d)}\le  r(\phi(P)) + r(\phi(I-P^k)) \le n,$$
where the last inequality follows from Lemma \ref{le:d-potent rank} (b). We conclude $r(\phi(P)) = r(P)$.
    
\item Since $P\perp Q$, by Lemma \ref{le:preceq properties} we have that $P+Q$ is again a $(k+1)$-potent and $P,Q \preceq P+Q$. Statements (b) and (c) imply $$\underbrace{\phi(P), \phi(Q)}_{\text{orthogonal}} \preceq \phi(P+Q) \stackrel{\text{Lemma }\ref{le:preceq properties} (b)}\implies\phi(P)+\phi(Q) \preceq \phi(P+Q).$$
    Finally, by the orthoadditivity of rank on $\pot_{k+1}(M_n)$ (Lemma~\ref{le:d-potent rank}~(b)) we have
    \begin{align*}
    r(\phi(P) + \phi(Q)) &= r(\phi(P)) + r(\phi(Q)) \stackrel{(e)}= r(P) + r(Q) = r(P+Q) \\
    &\leftstackrel{(e)}= r(\phi(P+Q)),
    \end{align*}
    so the equality follows by Lemma~\ref{le:d-potent rank}~(a).
    \item Similarly as in \cite[Lemma~3.2~(h)]{GogicTomasevic-AM} or  \cite[Lemma~3.4~(g)]{GogicTomasevic-LAMA} we have
 \begin{align*}
        \phi\left(\sum_{j=1}^r \lambda_j P_j\right) &= \phi\left(\left(\sum_{j=1}^r \lambda_j P_j\right) \circ \left(\sum_{\ell=1}^r  P_{\ell}\right)^k\right) = \phi\left(\sum_{j=1}^r \lambda_j P_j\right)  \circ \left(\phi\left(\sum_{\ell=1}^r P_{\ell}\right)\right)^k\\
        &\leftstackrel{(f)}= \phi\left(\sum_{j=1}^r \lambda_j P_j\right)  \circ \left(\sum_{\ell=1}^r \phi(P_{\ell})\right)^k \stackrel{(c)}= \phi\left(\sum_{j=1}^r \lambda_j P_j\right)  \circ \left(\sum_{\ell=1}^r \phi(P_{\ell})^k\right) \\
        &= \sum_{\ell=1}^r \left(\phi\left(\sum_{j=1}^r \lambda_j P_j\right)  \circ \phi(P_{\ell})^k\right)= \sum_{\ell=1}^r \phi\left(\left(\sum_{j=1}^r \lambda_j P_j\right)  \circ P_{\ell}^k\right) \\
        &= \sum_{\ell=1}^r \phi(\lambda_{\ell} P_{\ell}).
    \end{align*}
    \end{enumerate}
\end{proof}

In the sequel, $\K$ will denote the \emph{prime subfield} of $\F$, i.e.\ $\K$ is generated by the multiplicative identity of $\F$ (see e.g.\ \cite{DummitFoote}). Obviously $\K \cong \Q$ if $\chr(\F)=0$, or $\K \cong \Z/p\Z$ if $p=\chr(\F)>0$.

\begin{lemma}\label{le:h exists}
    Let $n\ge2$, and let $\phi:M_n\to M_n$ be a nonzero map which satisfies \eqref{eq:a certain product} and $\phi(0) = 0$. There exists a unique multiplicative map $\omega : \F \to \F$ such that 
\begin{equation}\label{eq:homogeneity}
        \phi(\lambda X) = \omega(\lambda) \phi(X), \quad  \text{ for all } \lambda \in \F \text{ and } X \in M_n.
    \end{equation}
    Moreover, there exists an invertible matrix $T \in M_n$ and $\varepsilon\in\Sph^1_k$ such that the map $$\psi:= T^{-1}\phi(\cdot)T$$ satisfies $$\psi(E_{jj})  = \varepsilon E_{jj}, \quad \text{ for each } j \in [n].$$
Further, fix some distinct $i,j \in [n]$. Then $\psi(E_{ij})$ is a nonzero scalar multiple of $E_{ij}$ or $E_{ji}$. If $\psi(E_{ij}) = \beta_{ij} E_{ij}$ for some $\beta_{ij} \in \F^\times$, then for each $x,y \in \F$ we have
        \begin{equation}\label{eq: Eii+Eij 1}
            \psi(E_{ii} + x E_{ij}) = \varepsilon E_{ii} +  \omega(x)  (\beta_{ij} E_{ij}), \qquad \psi(E_{jj} + y E_{ij}) = 
\varepsilon E_{jj} + \omega(y)   (\beta_{ij} E_{ij}).
        \end{equation}
If, on the other hand, $\psi(E_{ij}) = \gamma_{ij} E_{ji}$ for some $\gamma_{ij} \in \F^\times$, then
\begin{equation}\label{eq: Eii+Eij 2}
\psi(E_{ii} + x E_{ij}) = \varepsilon E_{ii} + \omega(x) (\gamma_{ij} E_{ji}), \qquad \psi(E_{jj} + y E_{ij}) = 
\varepsilon E_{jj} + \omega(y)  (\gamma_{ij} E_{ji}).
\end{equation}
Finally, the map $\omega : \F \to \F$ is a ring monomorphism. In particular, $\phi$ is $\K$-homogeneous.
\end{lemma}
\begin{proof}
In view of Lemma \ref{le:basic properties II} (c) and (e), $\phi(E_{11}), \ldots ,\phi(E_{nn})$ are mutually orthogonal rank-one $(k+1)$-potents and therefore can be simultaneously diagonalized (see e.g.\ \cite[Theorem~8 of \S6.5]{HoffmanKunze}). Hence, by passing to the map $T^{-1}\phi(\cdot)T$, for a suitable invertible matrix $T \in M_n$, without loss of generality we can assume that 
\begin{equation}\label{eq:phi fiksira dijagonalne}
\phi(E_{jj})  = \varepsilon_j E_{jj}, \quad \text{ for some $\varepsilon_j \in \Sph^1_k$ for each } j \in [n].
\end{equation}
Note that by Lemma \ref{le:basic properties II} (f), for each diagonal idempotent $P=\sum_{j\in S}E_{jj}$, where $S\subseteq[n]$, we have
    \begin{equation}\label{eq:diagonal tripotents}
        \phi(P)=\sum_{j\in S}\varepsilon_jE_{jj}.
    \end{equation}
By Lemma \ref{le:zero map}, we have 
\begin{equation}\label{eq:lambda times ejj}
    \phi(\lambda  E_{jj}) \ne 0, \quad\text{ for all } j \in [n] \text{ and } \lambda \in \F^\times.
\end{equation} Further, for each $X \in M_n$ and $S \subseteq [n]$ we have
\begin{equation}\label{eq:preserves support}
    \supp X \subseteq S \times S \implies \supp \phi(X) \subseteq S \times S.
\end{equation}
Indeed, denote the diagonal idempotent 
$P := \sum_{j \in [n]\setminus S} E_{jj}$ and note that a matrix $X\in M_n$ is supported in $S \times S$ if and only if $XP=PX=0$. In that case, obviously  $X \circ P^k = 0$, so 
    $$0 = \phi(X \circ P^k) = \phi(X) \circ \phi(P)^k \stackrel{\eqref{eq:diagonal tripotents}}= \phi(X) \circ P$$
        and hence Lemma \ref{le:Jordan product calculations} (a) implies the claim.
\smallskip

Let $j \in [n]$ and $\lambda \in \F^\times$. Then
    $$\phi(\lambda E_{jj}) = \phi((\lambda E_{jj}) \circ E_{jj}^k) = \phi(\lambda E_{jj}) \circ \phi(E_{jj})^k \stackrel{\eqref{eq:phi fiksira dijagonalne}}= \phi(\lambda E_{jj}) \circ E_{jj}.$$
    In view of Lemma \ref{le:Jordan product calculations} (b) we have $$\phi(\lambda E_{jj}) = E_{jj}\phi(\lambda E_{jj})E_{jj} = \phi(\lambda E_{jj})_{jj} E_{jj} \stackrel{\eqref{eq:phi fiksira dijagonalne}}= (\varepsilon_j^{-1}\phi(\lambda E_{jj})_{jj})  \phi(E_{jj}) .$$
    Since $\phi(0) = 0$, it follows that there exists a unique map $\omega_j : \F \to \F$ such that \begin{equation}\label{eq:definition of omegaj}
        \phi(\lambda E_{jj}) = \omega_j(\lambda)\phi(E_{jj}), \quad \text{ for all } \lambda \in \F.
    \end{equation}
    Further, by \eqref{eq:lambda times ejj} clearly $\omega_j^{-1}(\{0\}) = \{0\}$. Fix distinct $i,j \in [n]$. For $\lambda \in \F^\times$ by \eqref{eq:phi fiksira dijagonalne} we have
\begin{align}\label{eq:1/2 lambda racun}
\omega_i(\lambda)^k \phi(E_{ij}) \circ E_{ii} &= \phi(E_{ij}) \circ \phi(\lambda E_{ii})^k = \phi(E_{ij} \circ (\lambda E_{ii})^k) \\
&= \phi\left(\frac12\lambda^k E_{ij}\right) = \phi(E_{ij} \circ (\lambda E_{jj})^k) \nonumber \\
&= \phi(E_{ij}) \circ \phi(\lambda E_{jj})^k = \omega_j(\lambda)^k \phi(E_{ij}) \circ E_{jj}. \nonumber
\end{align}
Note that \eqref{eq:preserves support} implies that $\supp \phi( E_{ij}) \subseteq \{i,j\}\times\{i,j\}$ so we can denote
$$\phi (E_{ij}) = \sum_{(r,s) \in \{i,j\} \times \{i,j\}}\alpha_{rs} E_{rs}$$
for some $\alpha_{rs} \in \F$. We have
\begin{align*}
    &\omega_i(\lambda)^k\left(\alpha_{ii}E_{ii} + \frac12 \alpha_{ij}E_{ij} + \frac12 \alpha_{ji}E_{ji}\right) = \omega_i(\lambda)^k \phi(E_{ij}) \circ E_{ii} = \omega_j(\lambda)^k \phi(E_{ij}) \circ E_{jj} \\
    &\qquad= \omega_j(\lambda)^k\left(\alpha_{jj}E_{jj} + \frac12 \alpha_{ij}E_{ij} + \frac12 \alpha_{ji}E_{ji}\right).
\end{align*}
Hence, as $\omega_i(\lambda), \omega_j(\lambda) \ne 0$, it follows $\alpha_{ii} = \alpha_{jj} = 0$. Further, note that
$$0 = \phi(0) = \phi(E_{ij}^{k+1}) = \phi(E_{ij} \circ E_{ij}^k) = \phi(E_{ij}) \circ \phi(E_{ij})^k = \phi(E_{ij})^{k+1},$$
so that $\phi(E_{ij})$ is a nonzero nilpotent (Lemma \ref{le:zero map}). Therefore, we have
$$
    \phi(E_{ij}) \propto E_{ij} \quad \text{ or } \quad \phi(E_{ij}) \propto E_{ji}.
$$
It also follows that $\omega_i(\lambda)^k = \omega_j(\lambda)^k$ for all $\lambda \in \F^\times$, whence we conclude $\omega_i(\cdot)^k = \omega_j(\cdot)^k$. Returning back to the equality \eqref{eq:1/2 lambda racun}, we now obtain
\begin{equation}\label{eq:scalar times Eij}
    \phi\left(\frac12\lambda^k E_{ij}\right) = \frac12 \omega_i(\lambda)^k \phi(E_{ij})
\end{equation}
for all $\lambda\in\F^\times$.
Without loss of generality assume that $$\phi(E_{ij}) = \beta_{ij} E_{ij}, \quad\text{ for some } \beta_{ij} \in \F^\times$$ as in the other case we can pass to the map $\phi(\cdot)^t$.  We now proceed with the proof of equalities \eqref{eq: Eii+Eij 1} together with the required properties of $\omega$. Fix some $x \in \F^\times$. By \eqref{eq:preserves support} we see that
    $$\supp \phi(E_{ii} + x^k E_{ij}) \subseteq \{i,j\} \times \{i,j\}.$$
    Denote $$\phi (E_{ii} + x^k E_{ij}) = \sum_{(r,s) \in \{i,j\} \times \{i,j\}}\gamma_{rs} E_{rs}$$
    for some $\gamma_{rs} \in \F$. A similar calculation as before gives
     \begin{align*}
        \frac12\omega_i\left(x\right)^k (\beta_{ij} E_{ij}) \,\,&\leftstackrel{\eqref{eq:scalar times Eij}}= \phi\left(\frac12 x^k E_{ij}\right) = \phi((E_{ii} + x^k E_{ij}) \circ E_{jj}^k)\stackrel{\eqref{eq:phi fiksira dijagonalne}}= \phi(E_{ii} + x^k E_{ij}) \circ E_{jj} \\
        &= 
\frac12 \gamma_{ij}E_{ij} + \frac12 \gamma_{ji}E_{ji} + \gamma_{jj} E_{jj}.
    \end{align*}
    First we conclude $\gamma_{ji} = \gamma_{jj} = 0$ and $\gamma_{ij} =
\omega_i\left(x\right)^k \beta_{ij} \ne 0$. Since $\phi(E_{ii} + x^k E_{ij})$ is a $(k+1)$-potent (Lemma \ref{le:basic properties II} (a)), we also have
\begin{align*}
\gamma_{ii} E_{ii} + \gamma_{ij} E_{ij} &= \phi (E_{ii} + x^k E_{ij}) = \phi (E_{ii} + x^k E_{ij})^{k+1} = (\gamma_{ii} E_{ii} + \gamma_{ij} E_{ij})^{k+1}\\
&= \gamma_{ii}^{k+1} E_{ii} + \gamma_{ii}^k \gamma_{ij} E_{ij}.
\end{align*}
By comparing the coefficients, it follows that $\gamma_{ii} \in \Sph^1_{k}$. In addition,
\begin{align*}
    \gamma_{ii}E_{ii} + \gamma_{ij}E_{ij} &= \phi(E_{ii}+x^kE_{ij}) = \phi((E_{ii}+x^k E_{ij})^k \circ (E_{ii} + E_{jj})) \\
    &\leftstackrel{\eqref{eq:diagonal tripotents}}= (\gamma_{ii}E_{ii} + \gamma_{ij}E_{ij})^k \circ (\varepsilon_i E_{ii} + \varepsilon_j E_{jj})\\
    &= (\gamma_{ii}^kE_{ii} + \gamma_{ii}^{k-1}\gamma_{ij}E_{ij}) \circ (\varepsilon_i E_{ii} + \varepsilon_j E_{jj})\\
    &= (E_{ii} + \gamma_{ii}^{-1}\gamma_{ij}E_{ij}) \circ (\varepsilon_i E_{ii} + \varepsilon_j E_{jj})\\
    &= \varepsilon_i E_{ii} + \gamma_{ii}^{-1}\gamma_{ij}\left(\frac{\varepsilon_i + \varepsilon_j}2\right)E_{ij}.
\end{align*}
A direct comparison implies $\varepsilon_i = \varepsilon_j = \gamma_{ii}$, so
\begin{equation}\label{eq:partial idempotent with x}
    \phi(E_{ii}+x^kE_{ij}) = \varepsilon_i E_{ii} + \omega_i(x)^k(\beta_{ij} E_{ij}).
\end{equation}
Since $i,j$ were arbitrary, denote 
$$
\varepsilon := \varepsilon_1 = \cdots = \varepsilon_n.
$$
Further, for $\lambda \in \F^\times$ we have
\begin{align*}
\omega_i(\lambda^k)(\varepsilon E_{ii}) \quad\, &\leftstackrel{\eqref{eq:definition of omegaj}, \eqref{eq:phi fiksira dijagonalne}}= \phi(\lambda^k E_{ii}) = \phi((\lambda E_{ii})^k \circ E_{ii}) = \phi(\lambda E_{ii})^k \circ \phi(E_{ii})\\
&\leftstackrel{\eqref{eq:definition of omegaj}, \eqref{eq:phi fiksira dijagonalne}}= \omega_i(\lambda)^k (\varepsilon E_{ii}),
\end{align*}
implying $\omega_i(\lambda^k) = \omega_i(\lambda)^k$. As an analogous relation holds for $j$, from $\omega_i(\cdot)^k = \omega_j(\cdot)^k$ and the surjectivity of the map $\lambda \mapsto \lambda^k$ on $\F$, it follows that
$$\omega_i(\cdot) =  \omega_j(\cdot), \quad \text{ for all }i,j \in [n].$$
We conclude that there exists a unique globally defined map 
$$
\omega = \omega_1 = \cdots = \omega_n: \F \to \F
$$ such that
\begin{equation}\label{eq:semihomogeneous}
    \phi(\lambda E_{jj})  = \omega(\lambda)\phi(E_{jj}) = \omega(\lambda)(\varepsilon E_{jj}), \quad \text{ for all } \lambda \in \F \text{ and } j \in [n].
\end{equation}
Fix again some $i \in [n]$. For $\lambda, \mu \in \F$, using the above equality we have
       \begin{align*}
    \omega(\lambda^k\mu)\phi(E_{ii}) &= \phi((\lambda^k\mu) E_{ii}) = \phi((\lambda E_{ii})^k \circ (\mu E_{ii})) = \phi(\lambda E_{ii})^k \circ \phi(\mu E_{ii})\\
    &\leftstackrel{\eqref{eq:phi fiksira dijagonalne}}= \omega(\lambda)^k\omega(\mu)\phi(E_{ii}),
    \end{align*}
    which implies $\omega(\lambda^k\mu) = \omega(\lambda)^k\omega(\mu)$. Since $\omega$ commutes with $k$-th powers, the surjectivity of the $k$-th power map on $\F$ allows us to conclude that $\omega$ is indeed multiplicative.
   
    Now we prove \eqref {eq:homogeneity}. For $\lambda \in \F$ we have \begin{align*}
    \phi(\lambda I)&= \phi\left(\sum_{j \in [n]} \lambda E_{jj} \right)\stackrel{\text{Lemma } \ref{le:basic properties II} \text{ (g)}}= \sum_{j \in [n]} \phi(\lambda E_{jj})  \stackrel{\eqref{eq:semihomogeneous}}=\sum_{j \in [n]}\omega(\lambda) \phi(E_{jj})\\
    &\leftstackrel{\text{Lemma } \ref{le:basic properties II} \text{ (f)}}= \omega(\lambda) \phi(I).
    \end{align*}
Now, for arbitrary $X \in M_n$ and $\lambda \in \F$ we have
    \begin{align*}
        \phi(\lambda^k X) &= \phi(X \circ (\lambda I)^k) = \phi(X) \circ \phi(\lambda I)^k \\
        &= \omega(\lambda)^k (\phi(X) \circ \phi(I)^k) = \omega(\lambda^k) (\phi(X) \circ \phi(I)^k) \\
        &= \omega(\lambda^k) \phi(X).
    \end{align*}
We conclude \eqref{eq:homogeneity} again by invoking the surjectivity of the $k$-th power on $\F$. Overall, from \eqref{eq:partial idempotent with x} we read off that
    $$\phi(E_{ii} + x E_{ij}) = \varepsilon E_{ii} +  \omega\left(x\right) (\beta_{ij} E_{ij}), \quad \text{ for all }x \in \F \text{ and distinct }i,j \in [n].$$
    In an analogous way we arrive at the equality
    $$\phi(E_{jj} + y E_{ij}) = 
\varepsilon E_{jj} +  \omega\left( y\right)  (\beta_{ij} E_{ij}), \quad \text{ for all }y \in \F \text{ and distinct }i,j \in [n].$$
Finally, assume that $n \ge 2$ and fix some distinct $i,j \in [n]$. For all $x,y \in \F$ we have
    \begin{align*}
        \omega\left(\frac{x+y}2\right) \phi(E_{ij}) \, &\leftstackrel{\eqref{eq:homogeneity}}= \phi\left(\frac{x+y}2 E_{ij}\right) \\
        &= \phi((E_{ii} + x E_{ij}) \circ (E_{jj} + y E_{ij})^k) \\
        &= \phi(E_{ii} + x E_{ij}) \circ \phi(E_{jj} + y E_{ij})^k \\
        &= \left(\varepsilon E_{ii} + \omega\left(x\right)  (\beta_{ij} E_{ij})\right)  \circ \left(\varepsilon E_{jj} +  \omega\left( y\right) (\beta_{ij} E_{ij})\right)^k\\
        &= \left(\varepsilon E_{ii} + \omega\left(x\right)  (\beta_{ij} E_{ij})\right)  \circ \left( E_{jj} +  \varepsilon^{-1}\omega\left(y\right) (\beta_{ij} E_{ij})\right)\\
        &= \frac12\left(\omega\left( x\right) + \omega\left( y\right)\right) (\beta_{ij} E_{ij})\\
        &= \frac12\left(\omega\left( x\right) + \omega\left( y\right)\right)\phi(E_{ij})
    \end{align*}
    and hence
    \begin{equation}\label{eq:1/2 multiplicative}
        \omega\left(\frac{x+y}2\right) =  \frac12\left(\omega\left( x\right) + \omega\left( y\right)\right).
    \end{equation}
Plugging in $y = 0$ yields $\omega\left(\frac12 x\right) = \frac12 \omega(x)$ for all $x \in \F$. Returning back to \eqref{eq:1/2 multiplicative} we directly obtain the additivity of $\omega$. This finally closes the proof of the lemma.
\end{proof}

In the sequel, we will continue to use the map $\omega : \F \to \F$ from \eqref{eq:homogeneity} without further reference to Lemma~\ref{le:h exists}. 

\begin{lemma}\label{le:preserves triple product}
    Let $\phi : M_n \to M_n, n \ge 2$, be a map which satisfies \eqref{eq:a certain product} such that $\phi(0) = 0$ and $\phi(I) = I$. Then:
    \begin{itemize}
    \item $\phi$ preserves $k$-th powers, i.e.\ $\phi(X^k) = \phi(X)^k$ for all $X \in M_n$,
        \item $\phi(\idem(M_n)) \subseteq \idem(M_n)$,
        \item $\phi(PX^k P) = \phi(P)\phi(X)^k\phi(P)$ for all $X \in M_n$ and $P \in \idem(M_n)$.
    \end{itemize}
\end{lemma}
\begin{proof}
For any $X \in M_n$ we have
$$\phi(X^k) = \phi(X^k \circ I) = \phi(X)^k \circ \phi(I) = \phi(X)^k.$$
Fix some $P \in \idem(M_n)$ and denote $P^\perp := I-P$. Clearly $P \perp P^\perp$ so by Lemma \ref{le:basic properties II}  (f) we have
$$I = \phi(I) = \phi(P + P^\perp) = \phi(P) + \phi(P^\perp) \, \implies \, \phi(P^\perp) = I-\phi(P).$$
By (c) of  the same lemma we obtain $\phi(P^\perp) \perp \phi(P)$, so 
$$0 = \phi(P)\phi(P^\perp) = \phi(P)(I-\phi(P)) = \phi(P)-\phi(P)^2,$$
which implies that $\phi(P)$ is an idempotent. Hence, $\phi$ maps idempotents to idempotents. We also see that $\phi(P^\perp) = \phi(P)^\perp.$

\smallskip

Fix some $X \in M_n$. One easily verifies the equality
\begin{equation}\label{eq:PXP}
((P-P^\perp) \circ X^k) \circ P = PX^kP.
\end{equation}
We also have
\begin{equation}\label{eq:P-Pperp}
    \phi(P-P^\perp) \stackrel{\text{Lemma }\ref{le:basic properties II} (g)}= \phi(P) + \phi(-P^\perp)  \stackrel{\text{Lemma \ref{le:h exists}}}= \phi(P) - \phi(P)^\perp.
\end{equation}
    Hence, as $P^k = P$ and $\phi(P)^k=\phi(P)$, we have
    \begin{align*}
        \phi(PX^kP)\,\, &\leftstackrel{\eqref{eq:PXP}}= \phi\left(((P-P^\perp) \circ X^k) \circ P^k\right) = \left(\phi(P-P^\perp) \circ \phi(X)^k\right) \circ \phi(P)^k \\
        &\leftstackrel{\eqref{eq:P-Pperp}}= \left((\phi(P) - \phi(P)^\perp) \circ \phi(X)^k \right) 
        \circ \phi(P) \\
        &\leftstackrel{\eqref{eq:PXP}}= \phi(P)\phi(X)^k \phi(P).
    \end{align*}
\end{proof}

\begin{proof}[Proof of Theorem \ref{thm:main}]
The zero map is covered by (a), because $0$ is a $(k+1)$-potent. We therefore assume that $\phi$ is not the zero map. Since $\phi(0)$ is a $(k+1)$-potent (Lemma~\ref{le:basic properties II} (a)), without loss of generality we can assume that $\phi(0) = \diag(Q,0_{n-r})$ for some invertible $(k+1)$-potent $Q \in M_r$ where $0 \le r \le n$. Indeed, any $0$-blocks in the Jordan form of a $(k+1)$-potent, if present, are necessarily one-dimensional.

Assume that $r \ge 1$. Then we claim that $\phi$ is the constant map equal to $\phi(0)$. Indeed, as $Q^k = I_r$, first note that for all $X \in M_n$ we have
$$\diag(Q,0_{n-r}) = \phi(0) = \phi(X \circ 0^k) = \phi(X) \circ \phi(0)^k = \phi(X) \circ \diag(I_r,0_{n-r}),$$
which easily implies that 
\begin{equation}\label{eq:X to Y}
    \phi(X) = \diag(Q, \psi(X)),
\end{equation}
for some uniquely determined matrix $\psi(X) \in M_{n-r}$. In particular, if $r = n$, it follows that $\phi$ is the constant map globally equal to the matrix $Q$. Otherwise, it makes sense to consider the map $\psi : M_n \to M_{n-r}$ defined by \eqref{eq:X to Y}, which satisfies \eqref{eq:a certain product} and $\psi(0) = 0$. Since $n-r < n$, Lemma \ref{le:basic properties II} (d) implies that $\psi$ must be the zero map, and therefore $\phi(X) = \phi(0)$ for all $X \in M_n$.

Suppose now that $r = 0$, i.e.\ that $\phi(0) = 0$. We claim that $\phi$ takes the form given in \eqref{eq:form of phi}, and as a result, is a nonzero additive map. As in Lemma \ref{le:h exists} and its proof (by passing to the map $\varepsilon^{-1}T^{-1}\phi(\cdot)T$, if necessary), without loss of generality we can assume that
$$\phi(E_{jj}) = E_{jj}, \quad \text{ for all }j \in [n].$$
Consequently, by Lemma \ref{le:basic properties II} (f),
$$
   \phi(P) = P, \quad\text{ for all } P \in \idem(M_n) \cap \ca{D}_n. 
$$
In particular, $\phi(I) = I$ which by Lemma \ref{le:preserves triple product} implies that $\phi$ maps idempotents to idempotents. Again, by Lemma \ref{le:h exists}, for all distinct $i,j \in [n]$ we also have
$$\phi(E_{ij}) \propto E_{ij} \quad \text{ or } \quad \phi(E_{ij}) \propto E_{ji}.$$
We claim that the same option holds throughout. Consider the set
$$\ca{S} := \{(r,s)\in [n]^2 \setminus \Delta_n : \phi(E_{rs}) \propto E_{rs}\}.$$
We claim that $\ca{S} = \emptyset$ or it satisfies the assumptions of Lemma \ref{le:contains entire CxC}. Assume that $\ca{S}$ is nonempty and fix some $(i,j) \in \ca{S}$. Denote $\phi(E_{ij}) = \beta_{ij} E_{ij}$ for some $\beta_{ij} \in \F^\times$. We first claim that $(j,i) \in \ca{S}$ as well. Suppose the contrary, that $\phi(E_{ji}) = \gamma_{ji} E_{ij}$ for some $\gamma_{ji} \in \F^\times$. Then 
\begin{align*}
    \phi\left(\frac12(E_{ii}+E_{jj}+ E_{ij})\right) &= \phi((E_{jj}+E_{ji})^k \circ E_{ij}) = \phi(E_{jj}+E_{ji})^k \circ \phi(E_{ij}) \\
    &\leftstackrel{\eqref{eq: Eii+Eij 2}}= (E_{jj} + \gamma_{ji} E_{ij})^k \circ (\beta_{ij} E_{ij}) \\
    &= \frac12\beta_{ij} E_{ij}.
\end{align*}
Since $\left(\frac12(E_{ii}+E_{jj}+E_{ij})\right)^{k+1}\ne0$, whereas $\left(\frac12\beta_{ij}E_{ij}\right)^{k+1}=0$, preservation of $(k+1)$-st powers would give a nonzero matrix whose image is zero, contradicting Lemma~\ref{le:zero map}. The next objective is to show that $(i,p)\in\ca{S}$ for any $p \in [n]\setminus\{i\}$, and $(\ell,j)\in \ca{S}$ for any $\ell \in [n]\setminus\{j\}$.
     \begin{itemize}
        \item[--] Assume that $p \in [n]\setminus\{i,j\}$ and that $\phi(E_{ip}) = \gamma_{ip} E_{pi}$ for some $\gamma_{ip} \in \F^\times$. Then by Lemma \ref{le:h exists} we have
        \begin{align*}
   0 &= \phi(0) = \phi((E_{jj}+E_{ij})^k \circ E_{ip}) \\
   &= \phi(E_{jj}+E_{ij})^k \circ \phi(E_{ip}) \stackrel{\eqref{eq: Eii+Eij 1}}= (E_{jj} + \beta_{ij} E_{ij})^k \circ (\gamma_{ip} E_{pi}) \\
    &= \frac{\beta_{ij}\gamma_{ip}}{2} E_{pj},
\end{align*}
    is a contradiction as both scalars $\beta_{ij}$ and $\gamma_{ip}$ are nonzero. Consequently,  $\phi(E_{ip})  \propto E_{ip}$, that is $(i,p) \in \ca{S}$.
    \item[--] Assume that $\ell \in [n]\setminus\{i,j\}$ and that $\phi(E_{\ell j}) = \gamma_{\ell j} E_{j\ell}$ for some $\gamma_{\ell j} \in \F^\times$. Then, again by Lemma \ref{le:h exists}, 
    \begin{align*}
        0 &= \phi(0) = \phi((E_{ii}+E_{ij})^k \circ E_{\ell j}) \\
   &= \phi(E_{ii}+E_{ij})^k \circ \phi(E_{\ell j}) \stackrel{\eqref{eq: Eii+Eij 1}}= (E_{ii} + \beta_{ij} E_{ij})^k \circ (\gamma_{\ell j} E_{j\ell}) \\
    &= \frac{\beta_{ij}\gamma_{\ell j}}{2} E_{i\ell},
    \end{align*}
    is again a contradiction, so that $\phi(E_{\ell j})  \propto E_{\ell j}$ and hence $(\ell,j) \in \ca{S}$.
    \end{itemize}
    By Lemma \ref{le:contains entire CxC} it follows that $\ca{S} = [n]^2 \setminus \Delta_n$. On the other hand, if $\ca{S} = \emptyset$, by passing to the map $\phi(\cdot)^t$ we arrive at the same conclusion. Hence, in the sequel without loss of generality we can assume that there exists a map $g : [n]^2 \to \F^\times$ such that $g|_{\Delta_n} \equiv 1$ and
    $$\phi(E_{ij}) = g(i,j) E_{ij}, \quad \text{ for all }i,j \in [n].$$
    We claim that the map $g$ is transitive in the sense of \cite{Coelho}, i.e.\ it satisfies
        $$g(i,j)g(j,p) = g(i,p), \quad \text{ for all }(i,j), (j,p) \in [n]^2.$$ 
        It suffices to establish the above equality for pairs in $[n]^2 \setminus \Delta_n$.
        
        Fix $(i,j), (j,p) \in [n]^2 \setminus \Delta_n$. If $i \ne p$, then by Lemma \ref{le:h exists} we have
\begin{align*}
\frac12 g(i,p)E_{ip} &= \phi\left(\frac12 E_{ip}\right) =\phi((E_{ii}+E_{ij})^k \circ E_{jp}) = \phi(E_{ii}+E_{ij})^k \circ \phi(E_{jp}) \\
&\leftstackrel{\eqref{eq: Eii+Eij 1}}= (E_{ii}+g(i,j)E_{ij})^k \circ (g(j,p)E_{jp})\\
&= \frac12 g(i,j)g(j,p)E_{ip},
\end{align*}
which implies $g(i,p) = g(i,j)g(j,p)$.
On the other hand, if $i = p$, then by Lemma \ref{le:h exists} we have
\begin{align*}
\frac{1}{2}\phi\left(E_{ii} + E_{jj} + E_{ji}\right) &= \phi\left(\frac{1}{2}\left(E_{ii} + E_{jj} + E_{ji}\right)\right) = \phi((E_{ii}+E_{ij})^k \circ E_{ji}) \\
&= \phi(E_{ii}+E_{ij})^k \circ \phi(E_{ji}) \\
&\leftstackrel{\eqref{eq: Eii+Eij 1}}= (E_{ii}+ g(i,j)E_{ij})^k \circ (g(j,i)E_{ji})\\
&= \frac12 g(j,i) \Big(g(i,j)(E_{ii}+E_{jj}) + E_{ji}\Big).
\end{align*}
Furthermore, we have
\begin{align*}
     E_{ii} +\frac12g(j,i) E_{ji} \,\, &\leftstackrel{\eqref{eq: Eii+Eij 1}}= \phi\left(E_{ii} +\frac12 E_{ji}\right) =\phi\left((E_{ii} + E_{jj} + E_{ji}) \circ E_{ii}^k\right) \\
     &= \phi\left(E_{ii} + E_{jj} + E_{ji}\right) \circ \phi(E_{ii})^k\\
    &=  g(j,i) \Big(g(i,j)(E_{ii}+E_{jj}) + E_{ji}\Big) \circ E_{ii} \\
    &= g(j,i)\left(g(i,j) E_{ii} +\frac12 E_{ji}\right)
\end{align*}
We obtain 
$$
g(i,j)g(j,i) = 1,
$$
as desired. Following \cite{Coelho}, denote by $$g^* : M_n \to M_n, \qquad g^*(E_{ij}) := g(i,j)E_{ij}$$
the induced algebra automorphism of $M_n$. Note that $g^*$ is implemented via conjugation by the invertible diagonal matrix
$$
D := \operatorname{diag}\big(g(1,1), g(2,1), \ldots, g(n,1)\big) \in \ca{D}_n,
$$
i.e.\ $g^*(\cdot)=D(\cdot)D^{-1}$. Hence, by passing to the map $(g^*)^{-1}(\phi(\cdot)) = D^{-1}\phi(\cdot)D$, without loss of generality we can assume that 
$$\phi(E_{ij}) = E_{ij}, \quad \text{ for all }  (i,j) \in [n]^2.$$
In view of Lemma \ref{le:h exists}, denote by $\omega :\F \to \F$ the induced ring monomorphism that satisfies \eqref{eq:homogeneity}. We claim that 
\begin{equation}\label{eq:phi = omega}
        \phi(X) = \omega(X), \quad \text{ for all } X \in M_n.
\end{equation}
First assume that $X=[x_{ij}] \in M_n$ is of the form $X = Y^k$ for some $Y \in M_n$. Fix $(i,j) \in [n]^2$ and denote $\phi(X)=[\phi(X)_{ij}]$. If $i = j$, then by Lemma \ref{le:preserves triple product} we have 
    \begin{align*}
        \omega(x_{ii})E_{ii} &= \phi(x_{ii}E_{ii}) = \phi(E_{ii}Y^k E_{ii}) = \phi(E_{ii})\phi(Y)^k \phi(E_{ii}) = E_{ii}\phi(X) E_{ii} \\
        &= \phi(X)_{ii}E_{ii},
    \end{align*}
    which shows that
    \begin{equation}\label{eq:Yii}
        \phi(X)_{ii} = \omega(x_{ii}).
    \end{equation} Now assume $i \ne j$. Denote $P:= E_{ii}+E_{ji}$ and note that 
    $$
    PAP = (a_{ii}+a_{ij})P, \quad \text{ for all matrices $A=[a_{ij}] \in M_n$}.
    $$ 
    By \eqref{eq: Eii+Eij 1} we have $\phi(P) = P$ and hence
    \begin{align*}
         \omega(x_{ii}+x_{ij})P &=\phi((x_{ii}+x_{ij})P) = \phi(PXP) \stackrel{\text{Lemma }\ref{le:preserves triple product}}= \phi(P)\phi(X)\phi(P) \\
         &= P\phi(X)P = (\phi(X)_{ii}+\phi(X)_{ij})P.
    \end{align*}
    Therefore, by the additivity of $\omega$ and \eqref{eq:Yii} we obtain
    $$\omega(x_{ii})+\omega(x_{ij}) = \phi(X)_{ii}+\phi(X)_{ij} = \omega(x_{ii})+\phi(X)_{ij}.$$
    This implies $\phi(X)_{ij} = \omega(x_{ij})$, which shows \eqref{eq:phi = omega} for all matrices $X \in M_n(\F)$ which are a $k$-th power of some matrix. Using the notation of Lemma \ref{le:iterated R}, this means 
    $$\phi(X) = \omega(X), \quad \text{ for all } X \in \ca{R}_0^n.$$
    We now extend this conclusion to all matrices in $M_n$. Note that if $A,B \in M_n$ satisfy $\phi(A) = \omega(A)$ and $\phi(B) = \omega(B)$, then clearly
    $$\phi(A^k \circ B) = \phi(A)^k \circ \phi(B) = \omega(A)^k \circ \omega(B) = \omega(A^k \circ B).$$
    Again returning to the notation of Lemma \ref{le:iterated R}, this inductively implies that $$\phi(X) = \omega(X), \quad \text{ for all }  X\in \ca{R}^n.$$
   By the same lemma we have $\ca{R}^n=M_n$, which completes the proof of the necessity direction of the theorem.

Conversely, let $Q\in M_n$ satisfy $Q^{k+1}=Q$. The constant map $\phi(X)=Q$ satisfies \eqref{eq:a certain product}, since
$$
\phi(A)^k\circ\phi(B)=Q^k\circ Q=Q^{k+1}=Q=\phi(A^k\circ B).
$$
Now let $\rho$ be either of the maps
$$
\rho(X)=T\omega(X)T^{-1}
\qquad\text{or}\qquad
\rho(X)=T\omega(X)^tT^{-1},
$$
where $T$ is invertible and $\omega:\F\to\F$ is a ring monomorphism. Since $\omega\left(\frac12\right)=\frac12$, both entrywise application of $\omega$ and transposition preserve powers and the normalized Jordan product. Conjugation by $T$ does so as well. Hence
$$
\rho(A^k\circ B)=\rho(A)^k\circ\rho(B).
$$
If $\varepsilon^k=1$ and $\phi=\varepsilon\rho$, then
$$
\phi(A)^k\circ\phi(B)
=\varepsilon^{k+1}\bigl(\rho(A)^k\circ\rho(B)\bigr)
=\varepsilon\rho(A^k\circ B)
=\phi(A^k\circ B).
$$
Thus every map in (a) or (b) satisfies \eqref{eq:a certain product}.

\end{proof}
We conclude the paper with the following generalization of \cite[Corollary~3.5]{GogicTomasevic-AM}, which follows immediately from the first part of the proof of Theorem~\ref{thm:main}.
\begin{corollary} 
Let $m < n$. A map $\phi : M_n \to M_m$ satisfies \eqref{eq:a certain product} if and only if it is constant and equal to a fixed $(k+1)$-potent. 
\end{corollary}

\section*{Acknowledgments and Funding}

We are very grateful to the referee for the careful reading of our manuscript and for the helpful comments and suggestions, which have improved the paper.

\smallskip

The authors were supported by the European Union -- NextGenerationEU through the National Recovery and Resilience Plan 2021--2026 Institutional grants of the University of Zagreb Faculty of Science (IK IA 1.1.3. Impact4Math, PMF-CROFUND).

\end{document}